\documentclass[11pt,a4paper]{article}

\usepackage[T1]{fontenc}
\usepackage[utf8]{inputenc}
\usepackage{lmodern}
\usepackage[margin=1in]{geometry}
\usepackage{cite}
\usepackage{xcolor}
\usepackage{amsmath,amssymb,amsfonts,amsthm}
\usepackage{graphicx}
\usepackage{bm}
\usepackage{textcomp}
\usepackage{microtype}
\usepackage{enumitem}
\usepackage{needspace}
\usepackage[hidelinks]{hyperref}
\hypersetup{pdftitle={Necessary and Sufficient Criteria for Global PID Regulation of Nonlinear Uncertain Systems},
  pdfauthor={Cheng Zhao, Jingru Zhu, Lei Guo}}
\allowdisplaybreaks[2]
\newtheorem{theorem}{Theorem}[section]
\newtheorem{proposition}[theorem]{Proposition}
\newtheorem{lemma}[theorem]{Lemma}
\newtheorem{corollary}[theorem]{Corollary}
\theoremstyle{definition}
\newtheorem{definition}[theorem]{Definition}
\theoremstyle{remark}
\newtheorem{remark}[theorem]{Remark}

\newcommand{\R}{\mathbb R}
\newcommand{\Fcal}{\mathcal F}

\newcommand{\Bcal}{\mathcal B}

\newcommand{\dd}{\,\mathrm d}
\newcommand{\Opid}{\Omega_{\mathrm{pid}}}
\newcommand{\norm}[1]{\left\|#1\right\|}
\newcommand{\ip}[2]{\left\langle #1,#2\right\rangle}
\DeclareMathOperator{\col}{col}

\newcommand{\step}[1]{\par\smallskip\noindent\textbf{#1}\quad}

\makeatletter
\@addtoreset{equation}{section}
\makeatother

\title{Necessary and Sufficient Criteria for Global PID Regulation of Nonlinear Uncertain Systems}
\author{Cheng Zhao\thanks{State Key Laboratory of Mathematical Sciences,
Academy of Mathematics and Systems Science, Chinese Academy of Sciences,
Beijing 100190, China.}
\and Jingru Zhu\footnotemark[1]
\and Lei Guo\footnotemark[1]}
\date{}

\begin{document}
\maketitle

\begin{abstract}
Although classical PID control is widely used in engineering
practice, identifying all PID gains that guarantee global regulation
of nonlinear uncertain dynamics remains challenging.
For a class of second-order nonaffine systems under one-sided
derivative bounds, we establish necessary and sufficient conditions
for global PID regulation: a fixed PID controller globally regulates
the entire infinite-dimensional uncertainty class if and only if
it regulates a single boundary linear system.
Applying the Routh--Hurwitz criterion to this boundary linear system yields the exact admissible PID gain region.
To establish this equivalence, we prove a structured Markus--Yamabe
theorem for third-order integrator chains under uniform
derivative-ratio conditions.
Its proof combines a two-dimensional storage function constructed
via viscosity solutions of a Hamilton--Jacobi equation with
a scalar comparison argument. We further show that the same PID gain region
remains exact for multidimensional affine systems under
symmetry assumptions on the state partial Jacobians
and the input matrix.
\end{abstract}

\medskip
\noindent\textbf{Keywords:}
PID control, global regulation,
structured Markus--Yamabe theorem, nonlinear uncertain systems,  viscosity solutions

\section{Introduction}

Proportional--integral--derivative (PID) control is a linear feedback
 based on the output error, its integral, and its derivative.
Its simple structure and practical effectiveness have led to widespread
applications in process control, aerospace and automotive systems, and
robot manipulators
\cite{AstromHagglund2001,BouabdallahEtAl2004}.
The dynamics of these systems are typically nonlinear and subject to
uncertainty. Understanding the capability of linear PID feedback to globally
regulate nonlinear uncertain systems and deriving explicit gain
conditions remain important theoretical problems.

In practice, PID gains are often selected using empirical rules,
experiments, identified models, or local linearization
\cite{AstromEtAl1993}.
Local model information alone does not generally guarantee global
regulation over all setpoints and initial states.
Theoretical PID design has been studied extensively for linear systems,
including time-delay systems for which complete stabilizing gain
regions are available \cite{SilvaEtAl2005,OuZhangYu2009}.
For manipulators with compliant actuators, PID tracking control
has also been studied using singular perturbation methods
\cite{PanLiYu2019}.

For second-order affine nonlinear uncertain systems, Zhao and Guo
\cite{ZhaoGuo2017} established explicit sufficient PID gain conditions
and obtained an exact characterization for a subclass affine in the
velocity variable.
Subsequent work established sufficient conditions for global PID
regulation of nonaffine uncertain systems \cite{ZhaoGuo2022,ZhaoYuan2024}.
Sufficient conditions, however, may exclude admissible gains and thus limit the choices available for performance tuning.
This motivates the search for an exact characterization of all
PID gains that guarantee global regulation of a prescribed
nonlinear uncertainty class.

We consider scalar second-order nonaffine nonlinear systems
$\dot x_1=x_2$, $\dot x_2=f(x_1,x_2,u)$, where $f$ belongs to the infinite-dimensional uncertainty class
\[
 \Fcal_{L_1,L_2}
 :=\bigl\{f\in C^1(\R^3;\R):
 f_{x_1}\le L_1,\quad f_{x_2}\le L_2,\quad f_u\ge1
 \ \text{on }\R^3\bigr\},
\]
with prescribed constants $L_1,L_2>0$.
The positive lower bound on the input derivative is normalized to one.
The state-derivative bounds are one-sided: neither state derivative
is required to have a lower bound, and the input derivative has no
upper bound.
Thus the class includes systems that are nonaffine in the input
and need not be globally Lipschitz.
For each fixed PID gain triple, we prove that global regulation of
this entire nonlinear class is equivalent to global regulation of
the single boundary linear system
$f(x_1,x_2,u)=L_1x_1+L_2x_2+u$.
This system attains both state-derivative upper bounds and the
input-derivative lower bound. Its Routh--Hurwitz conditions give the exact admissible PID gain region
\[
 \Opid
 :=\bigl\{(k_p,k_i,k_d)\in\R^3:
 k_i>0,\quad k_p>L_1,\quad
 (k_p-L_1)(k_d-L_2)>k_i\bigr\}.
\]
These conditions depend only on the uncertainty bounds and guarantee
global regulation for every plant in the class.

The main difficulty in proving the exact characterization
is to show that the  gain conditions derived from the boundary linear
system also guarantee global asymptotic stability of the nonlinear
closed loop, without assuming global Lipschitz or
linear growth conditions of the plant dynamics.
For a class of planar affine systems, Zhao and Guo
\cite{ZhaoGuo2017} obtained necessary and sufficient PD gain
conditions using the two-dimensional Markus--Yamabe theorem,
provided that the prescribed setpoint is an equilibrium of the
uncontrolled plant.
In two dimensions, an everywhere Hurwitz Jacobian, together with an
equilibrium at the origin, implies global asymptotic stability
\cite{Gutierrez1995,ChenHeQin2001}.
This implication fails in general in dimensions three and higher,
even for polynomial vector fields \cite{CimaEtAl1997}.
Consequently, the planar argument cannot be extended directly after
introducing integral action.
After shifting the integral state to its equilibrium value, the PID
closed loop takes the companion form
\[
 \dot x=y,\qquad \dot y=z,\qquad \dot z=-F(x,y,z).
\]
The plant derivative bounds and the PID gains impose additional
relations among the derivatives of $F$.
The key question is whether these relations suffice to establish
global asymptotic stability for this structured family.

We answer this question by proving a structured Markus--Yamabe
theorem for third-order integrator chains under uniform
derivative-ratio conditions.
The main analytical step is to construct a two-dimensional storage
function using viscosity-solution theory for an auxiliary
Hamilton--Jacobi equation \cite{CrandallIshiiLions1992}.
Combined with a quadratic energy term, this function yields a
dissipation estimate that bounds two transformed state components,
while scalar comparison bounds the third.
Together, these estimates establish forward completeness,
boundedness, and global asymptotic stability.
The theorem is an independent stability result and supplies the
sufficiency argument for the exact PID gain characterization.

Finally, we extend the exact characterization to multidimensional
affine systems with scalar PID gains.
The two state partial Jacobians are assumed to be symmetric and
uniformly bounded above, and the constant input matrix is symmetric
positive definite with a known lower bound.
The proof uses a different mechanism: symmetry enables an explicit
variational construction of a positive definite, proper Lyapunov
function through integral potentials and finite-dimensional
maximization. Under these conditions, the same gain region and the same boundary
linear test remain valid in every dimension.
We also establish a limitation of this extension. If either
state-Jacobian symmetry assumption is removed while the
corresponding skew-symmetric part is left unrestricted,
no fixed choice of diagonal PID gain matrices guarantees
global regulation of the entire enlarged class. Thus, allowing different gains in different
coordinates does not remove this obstruction.

The paper is organized as follows.
Section~\ref{sec:main} presents the problem formulation and main results.
Section~\ref{sec:stability} develops the stability tools used to prove
the main results in Section~\ref{sec:proofs}.
Section~\ref{sec:conclusion} concludes the paper.
Appendices~\ref{app:HJ} and~\ref{app:storage} provide the
Hamilton--Jacobi result and the storage lemma, respectively.

\section{Problem Formulation and Main Results}
\label{sec:main}
In this section, we first formulate the global PID regulation problem
and characterize the exact gain region for scalar nonaffine systems
via a single boundary linear test.
We then consider the affine subclass and exhibit a fixed nonlinear
plant for which no admissible gains yield global exponential stability.
Finally, we extend the characterization to multidimensional affine
systems and clarify the role of
state-Jacobian symmetry.

\subsection{Scalar nonaffine systems: a single linear test}

Consider the second-order nonaffine nonlinear system
\begin{equation}\label{eq:plant}
 \dot x_1=x_2,\qquad \dot x_2=f(x_1,x_2,u),
 \qquad x_1,x_2,u\in\R,
\end{equation}
where $x_1$ and $x_2$ represent position and velocity, respectively,
$u$ is the control input, and $f\in C^1(\R^3;\R)$ is an unknown
nonlinear function.

The position--velocity structure in \eqref{eq:plant} follows from
Newton's second law and provides a basic model for motion control,
including actuated mass--spring--damper systems, single-axis
positioning mechanisms, and torque-actuated pendula
\cite{ZhaoGuo2022}.
Here, $f$ represents acceleration and incorporates uncertainty
in inertia, restoring forces, smooth friction, actuator
characteristics, and constant loads.
Nonlinear static actuator characteristics lead to nonaffine
input dependence.

Even when the precise dynamics are unknown, information about
stiffness, damping, and actuator sensitivity may provide bounds
on the partial derivatives of $f$ with respect to position,
velocity, and input.
These bounds define the uncertainty class and provide the
information used for PID gain design
\cite{ZhaoGuo2017,ZhaoGuo2022}.

\begin{definition}[Uncertainty class]\label{def:nonaffineclass}
For prescribed constants $L_1,L_2>0$, define
\begin{equation}\label{eq:nonaffineclass}
\begin{split}
 \Fcal_{L_1,L_2}:=
 \bigl\{f\in C^1(\R^3;\R):\;&
 f_{x_1}\le L_1,\quad f_{x_2}\le L_2,\quad f_u\ge1 \ \text{on }\R^3\bigr\}.
\end{split}
\end{equation}
Here, $f_{x_j}=\partial f/\partial x_j$ for $j=1,2$, and
$f_u=\partial f/\partial u$.
\end{definition}

In mechanical terms, positive values of $f_{x_1}$ and $f_{x_2}$
can be interpreted as \emph{anti-stiffness} and \emph{anti-damping},
respectively \cite{Krstic2017}.
With the other variables held fixed, these effects produce
acceleration changes that reinforce position and velocity
perturbations, respectively.
Thus, $L_1$ and $L_2$ bound the strength of these potentially
destabilizing effects.
These bounds are one-sided, allowing arbitrarily negative state
partial derivatives and hence strongly nonlinear restoring and
damping terms.

The condition $f_u\ge1$ specifies a known control direction and
a uniformly positive lower bound on control effectiveness.
Any known positive lower bound can be normalized to one by
input scaling.
No upper bound on $f_u$ is imposed, and the bias $f(0,0,0)$
is unrestricted, allowing nonlinear input dependence and
constant loads of unknown magnitude.

The class $\Fcal_{L_1,L_2}$ contains both linear and nonlinear
plants and is \emph{infinite-dimensional}.
For example, for any integer $m\ge1$, it contains
\[
 f(x_1,x_2,u)
 =\sum_{k=1}^{m}a_k x_1^{2k+1}
  +b\bigl(u+u^3\bigr),
 \qquad a_k\le0,\quad b\ge1.
\]
Indeed, $f_{x_1}\le0$, $f_{x_2}=0$, and
$f_u=b(1+3u^2)\ge1$.
Allowing arbitrary polynomial degrees gives admissible families of arbitrarily large dimension.
This example also shows that members of this class
need not be globally Lipschitz or satisfy a linear-growth
condition.

For a constant reference $y^*\in\R$, consider the PID controller
\begin{equation}\label{eq:pid}
 u=k_pe+k_i\eta+k_d\dot e,\qquad
 \dot\eta=e,\qquad e=y^*-x_1,
 \qquad (k_p,k_i,k_d)\in\R^3.
\end{equation}
Since the reference is constant, $\dot e=-x_2$.
The gains are to be selected using only the prescribed derivative
bounds, independently of the individual function $f$, the
reference, and the initial state.

\begin{definition}[Admissible PID gains]
\label{def:global}\label{def:admissible}
Let $\mathcal F\subset C^1(\R^3;\R)$ be an uncertainty class.
The PID controller~\eqref{eq:pid} with fixed gains
$(k_p,k_i,k_d)\in\R^3$ globally regulates the class of
systems~\eqref{eq:plant} defined by $f\in\mathcal F$
if, for every $f\in\mathcal F$, every constant reference
$y^*\in\R$, and every initial state
$(x_1(0),x_2(0),\eta(0))\in\R^3$, the closed-loop solution
exists for all $t\ge0$ and satisfies
\begin{equation}\label{eq:regulation}
 \lim_{t\to\infty}e(t)=0,\qquad
 \lim_{t\to\infty}\dot e(t)=0.
\end{equation}
A gain triple with this property is called admissible,
and the set of all such triples is called the admissible
PID gain region for this class of systems.
\end{definition}

We call the linear system with
\[
 f(x_1,x_2,u)=L_1x_1+L_2x_2+u\in \Fcal_{L_1,L_2}
\]
the \emph{boundary linear system}, since it attains both the
upper bounds on $f_{x_1}$ and $f_{x_2}$ and the lower bound on $f_u$
that define $\Fcal_{L_1,L_2}$.
The following theorem shows that this single linear system
determines the exact admissible PID gain region for the entire
infinite-dimensional nonlinear uncertainty class $\Fcal_{L_1,L_2}$.

\begin{theorem}\label{thm:scalar}
For the PID controller \eqref{eq:pid} with fixed gains
$(k_p,k_i,k_d)\in\R^3$, the following statements are equivalent:
\begin{itemize}
 \item[\rm i)] The PID controller globally regulates all
systems~\eqref{eq:plant} with $f\in\Fcal_{L_1,L_2}$.
 \item[\rm ii)] The PID controller globally regulates the boundary
linear system~\eqref{eq:plant} with
\[
 f(x_1,x_2,u)=L_1x_1+L_2x_2+u.
\]
 \item[\rm iii)] The PID gains satisfy
 $(k_p,k_i,k_d)\in\Opid$, where
 \begin{equation}\label{eq:Omega}
 \Opid:=\bigl\{(k_p,k_i,k_d)\in\R^3:
 k_i>0,\ k_p>L_1,\
 (k_p-L_1)(k_d-L_2)>k_i\bigr\}.
\end{equation}
\end{itemize}
\end{theorem}

\begin{remark}\label{rem:equilibrium}
Fix $f\in\Fcal_{L_1,L_2}$, $y^*\in\R$, and
$(k_p,k_i,k_d)\in\Opid$.
Since $f_u\ge1$, the equation $f(y^*,0,u^*)=0$ has a unique
solution $u^*\in\R$.
Define
\begin{equation}\label{eq:shiftdefinition}
 \eta^*=\frac{u^*}{k_i},\qquad
 \zeta=\col(\eta-\eta^*,e,\dot e).
\end{equation}
The proof of Theorem~\ref{thm:scalar} establishes global asymptotic
stability of the regulation equilibrium
$(x_1,x_2,\eta)=(y^*,0,\eta^*)$.
In particular,
\begin{equation}\label{eq:limits}
 (x_1(t),x_2(t),\eta(t),u(t))
 \longrightarrow (y^*,0,\eta^*,u^*).
\end{equation}
\end{remark}

\subsection{Scalar affine systems}
Consider now
\begin{equation}\label{eq:affineplant}
 \dot x_1=x_2,\qquad \dot x_2=f(x_1,x_2)+bu,
 \qquad x_1,x_2,u\in\R,
\end{equation}
where $b\ge1$ is an unknown constant and
\begin{equation}\label{eq:class}
 \Fcal^{\mathrm{aff}}_{L_1,L_2}:=\bigl\{f\in C^1(\R^2;\R):
 f_{x_1}\le L_1,\quad f_{x_2}\le L_2
 \quad\text{on }\R^2\bigr\}.
\end{equation}

The superscript $``\mathrm{aff}"$ indicates that this class is used
to describe systems affine in the control input. Global regulation of this class is understood in the sense of
Definition~\ref{def:global}, with the quantifiers extending to all
pairs $(f,b)\in\Fcal^{\mathrm{aff}}_{L_1,L_2}\times[1,\infty)$.

\begin{corollary}\label{cor:affine}
For each fixed gain triple $(k_p,k_i,k_d)$, the PID
controller~\eqref{eq:pid} globally regulates every pair
$(f,b)\in\Fcal^{\mathrm{aff}}_{L_1,L_2}\times[1,\infty)$ if and only if
$(k_p,k_i,k_d)\in\Opid$.
\end{corollary}

\begin{proof}[Proof of Corollary~\ref{cor:affine}]
For sufficiency, let $f\in\Fcal^{\mathrm{aff}}_{L_1,L_2}$ and $b\ge1$.
The function
\(
 \widetilde f(x_1,x_2,u)=f(x_1,x_2)+bu
\)
belongs to $\Fcal_{L_1,L_2}$.
Theorem~\ref{thm:scalar} therefore guarantees global regulation
whenever $(k_p,k_i,k_d)\in\Opid$.

For necessity, the affine uncertainty class contains the linear system
\[
 f(x_1,x_2)=L_1x_1+L_2x_2,\qquad b=1.
\]
By Theorem~\ref{thm:scalar}, global regulation of this system
requires $(k_p,k_i,k_d)\in\Opid$.
Hence this condition is also necessary for global regulation
of the entire affine class.
\end{proof}

For autonomous affine plants, Corollary~\ref{cor:affine}
strengthens \cite[Theorem~1]{ZhaoGuo2017} by identifying
the exact gain region while relaxing $|f_{x_2}|\le L_2$
to $f_{x_2}\le L_2$.
Proposition~1 of that paper obtains the same region $\Opid$
under the additional restriction
$f(x_1,x_2)=h(x_1)+a(x_1)x_2$.
The present corollary removes this restriction, allowing general
$C^1$ dependence on both state variables under the stated
derivative bounds.

The gains in $\Opid$ guarantee global asymptotic stability of
the regulation equilibrium. The following example shows that
this equilibrium need not be globally exponentially stable.

\begin{proposition}\label{prop:noglobalexp}
Consider the scalar affine plant~\eqref{eq:affineplant} with
$f(x_1,x_2)=-x_1^3$, $b\ge1$, and $y^*=0$.
For every PID gain triple in $\Opid$, the regulation equilibrium
is globally asymptotically stable but not globally exponentially
stable.
More precisely, no constants $M\ge1$ and $\lambda>0$ satisfy
\begin{equation}\label{eq:globalexponential}
 \norm{\zeta(t)}
 \le M e^{-\lambda t}\norm{\zeta(0)},\qquad t\ge0,
\end{equation}
for all initial plant states, even when $\eta(0)=0$.
\end{proposition}

\subsection{Coupled affine systems: exact gain region and symmetry}
Let $n\ge1$, and consider
\begin{equation}\label{eq:vectorplant}
 \dot x_1=x_2,\qquad \dot x_2=f(x_1,x_2)+Bu,
 \qquad x_1,x_2,u\in\R^n,
\end{equation}
under the vector version of \eqref{eq:pid}, with $y^*\in\R^n$.
The gains $k_p,k_i,k_d$ are scalars and are the same in every
coordinate. The input matrix is constant and belongs to
\begin{equation}\label{eq:Bclass}
 \Bcal_n:=\{B\in\R^{n\times n}: B=B^\top\succeq I_n\}.
\end{equation}
For symmetric matrices, $P\preceq Q$ means that $Q-P$ is positive semidefinite. Write $D_jf$ for the $n\times n$ partial Jacobian with respect to $x_j$. The vector uncertainty class is
\begin{equation}\label{eq:vectorclass}
\begin{split}
 \Fcal^{n}_{L_1,L_2}:=
 \bigl\{f\in C^1(\R^{2n};\R^n):\;
 D_jf=D_jf^\top\preceq L_jI_n,\ j=1,2,\ \forall (x_1,x_2)\in\R^{2n}\bigr\}.
\end{split}
\end{equation}
Global regulation is understood as in Definition~\ref{def:global},
with the quantifiers extending to every pair
$(f,B)\in\Fcal^{n}_{L_1,L_2}\times\Bcal_n$, every $y^*\in\R^n$,
and every initial state $(x_1(0),x_2(0),\eta(0))\in\R^{3n}$.

\begin{theorem}\label{thm:vector}
Let $n\ge1$, $L_1,L_2>0$, and fix $(k_p,k_i,k_d)\in\R^3$. Then the following statements are equivalent:
\begin{itemize}
  \item[\rm i)] The PID controller globally regulates all
 systems~\eqref{eq:vectorplant} with
 \[
 (f,B)\in\Fcal^{n}_{L_1,L_2}\times\Bcal_n.
 \]
 \item[\rm ii)] The PID controller globally regulates the boundary
 linear system~\eqref{eq:vectorplant} with
 \[
 f(x_1,x_2)=L_1x_1+L_2x_2,\qquad B=I_n.
 \]
 \item[\rm iii)] The PID gains satisfy  $(k_p,k_i,k_d)\in\Opid$.
\end{itemize}
\end{theorem}

Thus the same region $\Opid$ and the same boundary linear test remain
exact in every dimension $n\geq 1$ under the stated assumptions.

The uncertainty class~\eqref{eq:vectorclass} requires both partial
Jacobians $D_1f$ and $D_2f$ to be symmetric.
We next remove either symmetry condition while retaining the upper
bound on the symmetric part and leaving the skew-symmetric part
unrestricted. The following proposition shows that even diagonal
PID gains cannot guarantee global regulation of either enlarged
class.

\begin{proposition}\label{prop:symmetry}
Let $n\ge2$ and write $\operatorname{sym}A=(A+A^\top)/2$.
Define
\begin{align}
 \mathcal C_1
 &:=\bigl\{f\in C^1(\R^{2n};\R^n):
 \operatorname{sym}D_1f\preceq L_1I_n,\quad
 D_2f=D_2f^\top\preceq L_2I_n\bigr\},\label{eq:C1}\\
 \mathcal C_2
 &:=\bigl\{f\in C^1(\R^{2n};\R^n):
 D_1f=D_1f^\top\preceq L_1I_n,\quad
 \operatorname{sym}D_2f\preceq L_2I_n\bigr\},\label{eq:C2}
\end{align}
with all conditions imposed on $\R^{2n}$.
For each $j=1,2$, no fixed triple of real diagonal matrices
$K_p,K_i,K_d\in\R^{n\times n}$ in the PID controller
\[
 u=K_pe+K_i\eta+K_d\dot e,\qquad
 \dot\eta=e,\qquad e=y^*-x_1,
\]
guarantees global regulation for all $f\in\mathcal C_j$,
even with $B=I_n$ and $\eta(0)=0$.
\end{proposition}

\section{Global Stability of Structured Nonlinear Systems}
\label{sec:stability}

This section establishes two global asymptotic stability results
used in the sufficiency proofs of the main theorems.
Theorem~\ref{thm:MY} treats the third-order dynamics arising in
the scalar nonaffine problem and supplies the key step in
Theorem~\ref{thm:scalar}.
Its proof uses the two-dimensional storage function constructed
in Appendix~\ref{app:storage}.
Lemma~\ref{lem:normalform} treats the coupled affine dynamics
and supplies the key step in Theorem~\ref{thm:vector}.
Both results are stated independently of the PID problem.

\subsection{A structured Markus--Yamabe theorem}
\label{subsec:nonaffine-MY}

We first establish a global asymptotic stability result for
third-order integrators under derivative-ratio
conditions.

\begin{theorem}\label{thm:MY}
Let $F\in C^1(\R^3;\R)$. Suppose that there exist constants
$a,d,c>0$ with $ad>c$ such that
\begin{equation*}\tag{D}\label{eq:MYratios}
 F(0,0,0)=0,\qquad
 F_x\ge c,\qquad cF_y\ge aF_x,\qquad cF_z\ge dF_x
 \quad\text{on }\R^3.
\end{equation*}
Then the origin of
\begin{equation}\label{eq:MYnonaffine}
 \dot x=y,\qquad \dot y=z,\qquad \dot z=-F(x,y,z)
\end{equation}
is globally asymptotically stable.
\end{theorem}

In the application to scalar PID regulation, we take
\[
 a=k_p-L_1,\qquad d=k_d-L_2,\qquad c=k_i.
\]
With this choice, $ad>c$ is exactly the third inequality defining
$\Opid$. The derivative conditions in \eqref{eq:MYratios} follow
from the prescribed bounds on the plant derivatives, as verified
in the proof of Theorem~\ref{thm:scalar}.

\begin{proof}[Proof of Theorem~\ref{thm:MY}]
By \eqref{eq:MYratios}, $F_x,F_y,F_z>0$ and
\[
 F_yF_z\ge\frac{ad}{c^2}F_x^2
 \ge\frac{ad}{c}F_x>F_x.
\]
The characteristic polynomial of the Jacobian is
$\lambda^3+F_z\lambda^2+F_y\lambda+F_x$.
The cubic Routh--Hurwitz criterion  therefore shows that the Jacobian
Hurwitz at every point. It remains to prove global asymptotic stability.

\step{Step 1: Coordinate transformation.}
Since $ad>c$, we may choose $\frac ca<\mu<d$. Set
\[
  r=\frac c\mu,\qquad
 \delta=a-r,\qquad \varepsilon=d-\mu.
\]
Introduce the invertible linear change of coordinates
\begin{equation}
 q=(y,v,p),\qquad v=z,\qquad p=cx+\mu z,
 \label{eq:MYtransform}
\end{equation}
and define
\[
 B(y,v,p)=F\left(\frac{p-\mu v}{c},y,v\right)-ry.
\]
Then $B\in C^1(\R^3)$, $B(0)=0$, and the original system becomes
\begin{align}\label{standard}
 \dot y=v,\qquad \dot v=-ry-B(y,v,p),\qquad \dot p=-\mu B(y,v,p).
\end{align}
We shall prove that the origin of the standard form \eqref{standard} is globally asymptotically stable. By \eqref{eq:MYratios},
\[
 \begin{aligned}
 B_p&=\frac{F_x}{c}\ge1,\\
 B_v&=F_z-\mu B_p\ge(d-\mu)B_p=\varepsilon B_p,\\
 B_y&=F_y-r\ge aB_p-r
     =(a-r)B_p+r(B_p-1)\ge\delta B_p.
 \end{aligned}
\]
By Lemma~\ref{lem:storage} in Appendix~\ref{app:storage}, there exists a locally Lipschitz function $W:\R^2\to\R$ satisfying \eqref{eq:storage-value} and \eqref{eq:storage-match}, where $p_0(y)$ satisfies
\(
 B(y,0,p_0(y))=0
\).
The implicit function theorem gives $p_0\in C^1(\R)$; moreover, $B(0)=0$ and uniqueness of the zero imply $p_0(0)=0$. Define
\[
 E(y,v)=\frac12(v^2+ry^2),\qquad
 \kappa=\frac{\varepsilon\delta}{\delta+\mu\varepsilon}>0,
\]
and
\[
 V(y,v,p)=E(y,v)+W(y,p),\quad
 f_B(y,v,p)=\bigl(v,-ry-B(y,v,p),-\mu B(y,v,p)\bigr).
\]

\step{Step 2: Dissipation along every trajectory.}
Note that $V$ is locally Lipschitz and $V\ge E\ge0$.
At a point where $W$ is differentiable, write $\alpha=W_p$ and $h=W_y$. The storage function lemma gives
\[
 h=B(y,-\mu\alpha,p),\qquad \alpha h\ge\delta\alpha^2.
\]
Hence
\[
 \begin{aligned}
 \nabla V\cdot f_B
=vh-(v+\mu\alpha)B(y,v,p)=-(v+\mu\alpha)\bigl[B(y,v,p)-B(y,-\mu\alpha,p)\bigr]-\mu\alpha h.
 \end{aligned}
\]
Since $B_v\ge\varepsilon B_p\ge\varepsilon$,
\[
 (v+\mu\alpha)\bigl[B(y,v,p)-B(y,-\mu\alpha,p)\bigr]
 \ge\varepsilon(v+\mu\alpha)^2.
\]
Combining this inequality with $\alpha h\ge\delta\alpha^2$ and completing the square in $\alpha$, we obtain
\[
 \nabla V\cdot f_B\le-\varepsilon(v+\mu\alpha)^2-\mu\delta\alpha^2
 \le-\frac{\varepsilon\delta}{\delta+\mu\varepsilon}v^2=-\kappa v^2.
\]
Since $W$ is locally Lipschitz, Rademacher's theorem yields
\begin{equation}
 \nabla V(q)\cdot f_B(q)\le-\kappa v^2
 \qquad\text{almost everywhere in $\R^3$.}
 \label{eq:MYae}
\end{equation}
Since $f_B\in C^1$, each initial condition $q_{\mathrm{in}}\in \mathbb{R}^3$ determines a unique maximal solution $q(t)$, whose interval of existence is denoted by $[0,T_{\max})$. Fix $0<T<T_{\max}$ and choose a closed ball $K$ large enough to contain the trajectory segment
$q([0,T])$ and all points at distance at most \(1\) from this segment. For an open set $\Omega\subset\R^d$, let $C_c^\infty(\Omega)$ denote the space of smooth functions with compact support in $\Omega$. Choose a nonnegative kernel $\rho\in C_c^\infty(B_1(0))$ satisfying $\int_{\R^3}\rho=1$, and set
\[
 \rho_\nu(\xi)=\nu^{-3}\rho(\xi/\nu),\qquad
 V_\nu=\rho_\nu*V,\qquad g(q)=\kappa v^2
 \qquad(0<\nu<1).
\]
Because $V$ is locally Lipschitz, $V_\nu$ is smooth, converges locally uniformly to $V$, and has gradient equal to the convolution of the weak gradient of $V$ with $\rho_\nu$. By \eqref{eq:MYae}, at every point \(q\) on the trajectory segment,
\begin{equation}
 \begin{aligned}
 \nabla V_\nu(q)\cdot f_B(q)
 &\le-(\rho_\nu*g)(q)+\int_{\R^3}\rho_\nu(\xi)\nabla V(q-\xi)\cdot
     [f_B(q)-f_B(q-\xi)]\dd\xi\\
 &\le-g(q)+C_T\nu.
 \end{aligned}
 \label{eq:MYmollified}
\end{equation}
Here $C_T<\infty$ is independent of $\nu$: the gradient $\nabla V$ is essentially bounded on $K$, and $f_B,g\in C^1$ are Lipschitz on this closed ball, so both convolution errors are $O(\nu)$. Applying the chain rule to the smooth function $V_\nu$ along the trajectory and integrating, we obtain
\[
 V_\nu(q(t))+\kappa\int_s^t v(\tau)^2\dd\tau
 \le V_\nu(q(s))+C_T\nu(t-s)
 \qquad(0\le s\le t\le T).
\]
Letting $\nu\downarrow0$ and using the arbitrariness of $0<T<T_{\max}$, we conclude that
\begin{equation}
 V(q(t))+\kappa\int_s^t v(\tau)^2\dd\tau\le V(q(s))
 \qquad(0\le s\le t<T_{\max}).
 \label{eq:MYdissipation}
\end{equation}
In particular, $W\ge0$ and \eqref{eq:storage-value} imply
\begin{equation}
 E(y(t),v(t))+\kappa\int_0^t v(\tau)^2\dd\tau\le C_0
 \qquad(0\le t<T_{\max}),
 \label{eq:MYenergy}
\end{equation}
where
\[
 C_0=C_0(q_{\mathrm{in}}):=
 E(y_{\mathrm{in}},v_{\mathrm{in}})
 +\frac{[p_{\mathrm{in}}-p_0(y_{\mathrm{in}})]^2}{2\mu\varepsilon}.
\]

\step{Step 3: Forward completeness and Lyapunov stability.}
Fix an arbitrary initial condition. By \eqref{eq:MYenergy},
\begin{equation}
 |y(t)|\le\sqrt{2C_0/r},\qquad |v(t)|\le\sqrt{2C_0}
 \qquad(0\le t<T_{\max}).
 \label{eq:MYyvbound}
\end{equation}
For $C\ge0$, define
\[
 K_C=\{(y,v):E(y,v)\le C\},\qquad
 M_C=\max_{(y,v)\in K_C}|B(y,v,0)|<\infty.
\]

Since $B_p\ge1$, along the trajectory we have
\[
 \dot p+\mu b(t)p=-\mu B(y(t),v(t),0),\qquad
 b(t)=\int_0^1 B_p(y(t),v(t),\theta p(t))\dd\theta\ge1.
\]
Since $|B(y(t),v(t),0)|\le M_{C_0}$, the integrating factor method on any finite interval of existence gives
\begin{equation}
 \begin{aligned}
 |p(t)|\le e^{-\mu t}|p_{\mathrm{in}}|+(1-e^{-\mu t})M_{C_0}
 \le\max\{|p_{\mathrm{in}}|,M_{C_0}\}\qquad(0\le t<T_{\max}).
 \end{aligned}
 \label{eq:MYpbound}
\end{equation}
Together with \eqref{eq:MYyvbound}, this shows that the entire maximal trajectory remains in a compact set. The continuation theorem for ordinary differential equations therefore implies that all solutions of the standard form are forward complete and bounded.

We next show that these estimates also establish Lyapunov stability of the origin. Indeed, since $p_0(0)=0$ and $p_0$ is continuous,
\[
 C_0(q_{\mathrm{in}})\longrightarrow0\qquad(q_{\mathrm{in}}\to0).
\]
Furthermore, since
\[
 K_C\subset\bigl\{(y,v):|y|\le\sqrt{2C/r},\ |v|\le\sqrt{2C}\bigr\}
\]
and $B(0)=0$, continuity of $B$ at the origin implies
\(
 M_C\longrightarrow0\quad(C\downarrow0).
\)
Thus \eqref{eq:MYyvbound}--\eqref{eq:MYpbound} yield
\[
 \sup_{t\ge0}\|q(t)\|^2
 \le\frac{2C_0}{r}+2C_0+\max\{|p_{\mathrm{in}}|,M_{C_0}\}^2
 \longrightarrow0\qquad(q_{\mathrm{in}}\to0).
\]

\step{Step 4: Global attractivity.}
By Step 3, the standard form generates a continuous semiflow on $\R^3$, and every forward trajectory is bounded and hence relatively compact. The function $V$ is continuous and, by \eqref{eq:MYdissipation}, nonincreasing along every trajectory.

Suppose that $q(t)=(y(t),v(t),p(t))$ is a trajectory along which $V(q(t))$ is constant. By \eqref{eq:MYdissipation},
\[
 0\le\kappa\int_0^t v(s)^2\dd s
 \le V(q(0))-V(q(t))=0\qquad(t\ge0).
\]
Continuity of $v$ implies $v(t)\equiv0$. The equations of the standard form then give
\[
 y(t)\equiv Y,\qquad B(Y,0,p(t))=-rY.
\]
Since $B_p\ge1$, the function $p\mapsto B(Y,0,p)$ is strictly increasing, so $p(t)\equiv P$. Hence
\(
 0=\dot p=-\mu B(Y,0,P)=\mu rY,
\)
which implies $Y=0$. Moreover,
\(
 B(0,0,P)=0=B(0,0,0)
\)
and $B_p\ge1$ imply $P=0$. Thus the only complete trajectory along which $V$ is constant is the equilibrium at the origin. By LaSalle's invariance principle, every trajectory converges to the origin. Together with Step 3, this proves global asymptotic stability of the origin of the standard form. The invertible linear transformation \eqref{eq:MYtransform} then implies that the origin of the original system \eqref{eq:MYnonaffine} is also globally asymptotically stable.
\end{proof}

\subsection{A stability criterion for vector affine systems}
\label{subsec:vector-stability}

The vector problem requires a separate argument for the coupled
dynamics. Symmetry of the two partial Jacobians supplies
potential identities, and maximizing an auxiliary function yields
a continuous, positive definite, proper Lyapunov function.
The following normal form is the one needed in the proof of
Theorem~\ref{thm:vector}.

\begin{lemma}\label{lem:normalform}
Let $r,b,\delta>0$, $B=B^\top\succeq I_n$, and $N\in C^1(\R^{2n};\R^n)$ satisfy $N(0,0)=0$ and
\begin{equation}\label{eq:Nassumptions}
 D_yN=D_yN^\top\succeq\delta I_n,\qquad
 D_zN=D_zN^\top\succeq0
 \quad\text{on }\R^{2n}.
\end{equation}
Then the origin of
\begin{equation}\label{eq:normalform}
 \dot y=z,\qquad \dot z=-rBy+bs-N(y,z),\qquad
 \dot s=-bs+N(y,z)
\end{equation}
is globally asymptotically stable.
\end{lemma}

\begin{proof}
For $w,q\in\R^n$, define
\[
 K(w,q)=\int_0^1\ip{N(w+\theta q,-b\theta q)}q\,\dd\theta,
 \qquad
 H(w,q)=\int_0^1D_yN(w+\theta q,-b\theta q)^\top q\,\dd\theta.
\]
Then $K\in C^1$ and $\nabla_wK=H$. For fixed $w$, the map
$M_w(v)=N(w+v,-bv)$ has symmetric Jacobian
$DM_w=D_yN-bD_zN$. Therefore
\[
 \begin{aligned}
 \nabla_qK(w,q)
 &=\int_0^1\bigl[M_w(\theta q)
       +\theta DM_w(\theta q)^\top q\bigr]\dd\theta\\
 &=\int_0^1\frac{\dd}{\dd\theta}
       \bigl[\theta M_w(\theta q)\bigr]\dd\theta
 =N(w+q,-bq).
 \end{aligned}
\]
The lower bound on $D_yN$ also gives
$\ip q{H(w,q)}\ge\delta\norm q^2$.

Define
\[
 E(y,z)=\frac12\norm z^2+\frac r2y^\top By,
 \qquad
 L(y,z,s,q)=E(y,z)+K(y-q,q)-b\ip sq.
\]
We use the function
\begin{equation}\label{eq:boxedV}
\begin{aligned}
 V(y,z,s)&=\max_{q\in\R^n}L(y,z,s,q)\\
 &=\frac12\norm z^2+\frac r2y^\top By
 +\max_{q\in\R^n}\left\{
 \int_0^1\ip{N(y-(1-\theta)q,-b\theta q)}q\,\dd\theta
 -b\ip sq\right\}.
\end{aligned}
\end{equation}
To justify this definition, integrate the two matrix inequalities in
\eqref{eq:Nassumptions} along straight segments. They yield
\[
 \begin{aligned}
 \ip q{N(y-(1-\theta)q,-b\theta q)}
 &\le\ip q{N(y,-b\theta q)}
      -\delta(1-\theta)\norm q^2\\
 &\le\ip q{N(y,0)}-\delta(1-\theta)\norm q^2.
 \end{aligned}
\]
Consequently, with $\Phi(y,s,q)=K(y-q,q)-b\ip sq$,
\begin{equation}\label{eq:vectorcoercivity}
 \Phi(y,s,q)\le\ip{N(y,0)-bs}q-\frac\delta2\norm q^2.
\end{equation}
Thus the maximum is attained. Since $\Phi(y,s,0)=0$, every
maximizer satisfies
\[
 \norm{q_*}\le\frac{2}{\delta}\norm{N(y,0)-bs}.
\]
In a neighborhood of any fixed $(y,z,s)$, all maximizers therefore
belong to one compact ball. Maximizing the continuous function $L$
over this ball shows that $V$ is continuous.

Taking $q=0$ gives $V\ge E\ge0$. At the origin,
\eqref{eq:vectorcoercivity} gives $V(0,0,0)=0$.
If $(y,z)\ne(0,0)$, then $E>0$. If $y=z=0$ and $s\ne0$,
then $\nabla_q\Phi(0,s,0)=-bs\ne0$, so
$\max_q\Phi(0,s,q)>0$. Hence $V$ is positive definite.
It is also proper. Indeed, $V\le R$ bounds $y,z$ because
$B\succeq I_n$. On the resulting bounded set of $y$, continuity
gives a constant $D_R$ such that
$|K(y-q,q)|\le D_R$ whenever $\norm q=1$.
For $s\ne0$, the trial vector $q=-s/\norm s$ gives
$V\ge-D_R+b\norm s$. Thus $V\le R$ also bounds $s$, and
every sublevel set of $V$ is compact.

Fix any solution $\xi(t)=(y(t),z(t),s(t))$ on a compact interval
of its maximal existence interval, and augment it by
$\dot q=z+bq$. Set $w=y-q$, so $\dot w=-bq$.
Symmetry of $B$ gives $\dot E=-\ip z{\dot s}$.
The identities for $K$ imply
\begin{align}
 \frac{\dd}{\dd t}L(\xi(t),q(t))
 &=-\ip{z+bq}{N(y,z)-N(y,-bq)}
   -b\ip q{H(y-q,q)}\notag\\
 &\le-b\delta\norm q^2.
 \label{eq:vectorLdissipation}
\end{align}
Here $D_zN\succeq0$ implies
$\ip{z+bq}{N(y,z)-N(y,-bq)}\ge0$.

For $t_0<t_1$, choose any maximizer $q_1$ defining
$V(\xi(t_1))$. The affine map
\[
 q(t_0)\longmapsto q(t_1)
 =e^{b(t_1-t_0)}q(t_0)
  +\int_{t_0}^{t_1}e^{b(t_1-\tau)}z(\tau)\,\dd\tau
\]
is a bijection, so there is a unique auxiliary trajectory with
$q(t_1)=q_1$. Integrating \eqref{eq:vectorLdissipation} along it
and using $L(\xi(t_0),q(t_0))\le V(\xi(t_0))$ yields
\begin{equation}\label{eq:vectorendpoint}
 V(\xi(t_1))\le V(\xi(t_0))
 -b\delta\int_{t_0}^{t_1}\norm{q(t)}^2\,\dd t.
\end{equation}
In particular, $V$ is nonincreasing along every solution. Its compact
sublevel sets and the continuation theorem imply forward
completeness and boundedness.

We identify the trajectories along which $V$ is constant.
For any nontrivial interval of such a trajectory,
\eqref{eq:vectorendpoint} forces $q\equiv0$.
The auxiliary equation then gives $z\equiv0$, so $y$ is constant.
The equation for $\dot z$ gives $bs=rBy+N(y,0)$, so $s$ is also
constant. Adding the equations for $\dot z$ and $\dot s$ gives
$rBy=0$. Since $B\succ0$, this implies $y=0$ and then $s=0$.
Thus the origin is the only trajectory on which $V$ is constant.

Every solution has a nonempty compact invariant omega-limit set.
Continuity and monotonicity of $V$ make $V$ constant on this set,
which must therefore consist only of the origin. Every solution
converges to the origin. Finally, for any $\epsilon>0$, the minimum
of $V$ on the sphere $\norm\xi=\epsilon$ is positive.
Continuity at the origin gives a neighborhood where $V$ is smaller
than this minimum, and monotonicity prevents any trajectory starting
there from reaching the sphere. This proves Lyapunov stability and
completes the proof of global asymptotic stability.
\end{proof}

\Needspace{6\baselineskip}
\section{Proof of the Main Results}\label{sec:proofs}

\subsection{Scalar systems}
\begin{proof}[Proof of Theorem~\ref{thm:scalar}]
\emph{Step 1: \textup{(i)} $\Rightarrow$ \textup{(ii)}.}
The boundary linear system belongs to $\Fcal_{L_1,L_2}$,
so this implication follows directly from inclusion.

\emph{Step 2: \textup{(ii)} $\Rightarrow$ \textup{(iii)}.}
For the boundary linear system, the tracking error satisfies
\[
 \ddot e+(k_d-L_2)\dot e+(k_p-L_1)e+k_i\eta=-L_1y^*.
\]
Differentiating and using $\dot\eta=e$ gives
\begin{equation}\label{eq:linearerror}
 e^{(3)}+(k_d-L_2)\ddot e+(k_p-L_1)\dot e+k_ie=0.
\end{equation}
Since $L_1>0$, every initial triple
$(e(0),\dot e(0),\ddot e(0))=(r_0,r_1,r_2)\in\R^3$
can be realized with $\eta(0)=0$ by choosing
\begin{equation}\label{eq:realize}
 \begin{aligned}
 y^*&=-\frac{r_2+(k_p-L_1)r_0+(k_d-L_2)r_1}{L_1},\\
 x_1(0)&=y^*-r_0,\qquad x_2(0)=-r_1.
 \end{aligned}
\end{equation}
Thus, because \textup{(ii)} applies in particular to zero integral
initialization for every constant reference and plant initial state,
every solution of
\eqref{eq:linearerror} converges to zero.
Its characteristic polynomial
\begin{equation}\label{eq:polynomial}
 \lambda^3+(k_d-L_2)\lambda^2+(k_p-L_1)\lambda+k_i
\end{equation}
must therefore be Hurwitz.
The cubic Routh--Hurwitz criterion gives exactly
$(k_p,k_i,k_d)\in\Opid$, proving \textup{(iii)}.

\emph{Step 3: \textup{(iii)} $\Rightarrow$ \textup{(i)}.}
Assume \textup{(iii)} and fix $f\in\Fcal_{L_1,L_2}$
and $y^*\in\R$.
Since $f_u\ge1$, the map $u\mapsto f(y^*,0,u)$ is strictly
increasing and surjective, so it has a unique zero $u^*$.
Introduce the coordinates
\begin{equation}\label{eq:nonaffinetransform}
 x=\eta-u^*/k_i,\qquad y=e,\qquad z=\dot e.
\end{equation}
The closed-loop system becomes
\[
 \dot x=y,\qquad \dot y=z,\qquad \dot z=-F(x,y,z),
\]
where
\begin{equation}\label{eq:nonaffineF}
 F(x,y,z)
 =f(y^*-y,-z,u^*+k_ix+k_py+k_dz).
\end{equation}
Clearly, $F\in C^1(\R^3;\R)$ and $F(0,0,0)=0$.
Set
\begin{equation}\label{eq:nonaffinecoefficients}
 c=k_i,\qquad a=k_p-L_1,\qquad d=k_d-L_2.
\end{equation}
By \eqref{eq:Omega}, $a,d,c>0$ and $ad>c$.
With the derivatives of $f$ evaluated at the arguments
in \eqref{eq:nonaffineF}, we obtain
\begin{align*}
 F_x&=cf_u\ge c,\\
 cF_y-aF_x
 &=c\bigl(L_1f_u-f_{x_1}\bigr)\ge0,\\
 cF_z-dF_x
 &=c\bigl(L_2f_u-f_{x_2}\bigr)\ge0.
\end{align*}
Thus $F$ satisfies \eqref{eq:MYratios}.
Theorem~\ref{thm:MY} therefore establishes forward completeness
and global asymptotic stability of the origin.
Returning to the original variables gives \eqref{eq:limits}
for every initial state, and hence global regulation.
This proves \textup{(i)} and completes the equivalence.

\end{proof}

\begin{proof}[Proof of Proposition~\ref{prop:noglobalexp}]
The function $f(x_1,x_2)=-x_1^3$ belongs to $\Fcal^{\mathrm{aff}}_{L_1,L_2}$ because
\[
 \partial_{x_1}f=-3x_1^2\le L_1,\qquad \partial_{x_2}f=0\le L_2.
\]
Fix any gain triple in $\Opid$, any $b\ge1$, and take $y^*=0$. Set $\kappa_j=bk_j$ for $j=p,i,d$. In the shifted coordinates $x=\eta$, $y=e$, $z=\dot e$, the closed loop is
\begin{equation}\label{eq:cubiccounter}
 \dot x=y,\qquad \dot y=z,\qquad
 \dot z=-\kappa_i x-\kappa_p y-\kappa_d z-y^3.
\end{equation}
Its origin is globally asymptotically stable by Corollary~\ref{cor:affine}. We show that no constants $M\ge1$, $\lambda>0$ can give a global estimate
\begin{equation}\label{eq:globalexpfalse}
 \norm{\zeta(t)}\le M e^{-\lambda t}\norm{\zeta(0)},\qquad
 t\ge0,\quad\zeta(0)\in\R^3.
\end{equation}

For $R>0$, choose $x(0)=0$, $y(0)=R$, $z(0)=0$, so that $\eta(0)=0$ and $\norm{\zeta(0)}=R$. Introduce the rescaled time and states
\[
 \tau=Rt,\qquad
 X_R(\tau)=x(\tau/R),\quad
 Y_R(\tau)=\frac{y(\tau/R)}R,\quad
 Z_R(\tau)=\frac{z(\tau/R)}{R^2}.
\]
With primes denoting differentiation with respect to $\tau$, equation~\eqref{eq:cubiccounter} becomes
\begin{equation}\label{eq:rescaledcounter}
 \begin{gathered}
 X_R'=Y_R,\qquad Y_R'=Z_R,\\
 Z_R'=-Y_R^3-\frac{\kappa_d}{R}Z_R-\frac{\kappa_p}{R^2}Y_R
       -\frac{\kappa_i}{R^3}X_R,\\
 (X_R(0),Y_R(0),Z_R(0))=(0,1,0).
 \end{gathered}
\end{equation}
As $R\to\infty$, the vector fields converge in $C^1$ on compact sets to
\[
 X'=Y,\qquad Y'=Z,\qquad Z'=-Y^3,
 \qquad (X(0),Y(0),Z(0))=(0,1,0).
\]
Since $Z'(0)=-1$, choose a fixed sufficiently small $\tau_0>0$ such that $Z(\tau_0)<0$, and set $c_0=-Z(\tau_0)/2>0$. Continuous dependence of local solutions on parameters gives
\[
 Z_R(\tau_0)\longrightarrow Z(\tau_0),
\]
so $|Z_R(\tau_0)|\ge c_0$ for all sufficiently large $R$. Returning to the original variables,
\[
 \norm{\zeta(\tau_0/R)}\ge |z(\tau_0/R)|
 =R^2|Z_R(\tau_0)|\ge c_0R^2.
\]
But \eqref{eq:globalexpfalse} would imply
\[
 c_0R^2\le M e^{-\lambda\tau_0/R}R\le MR,
\]
which is impossible as $R\to\infty$. The nonlinearity, reference, and gains are fixed throughout; only the initial plant state varies. This proves the failure of global exponential stability even with zero integral initialization.
\end{proof}

\subsection{Coupled systems}

\begin{proof}[Proof of Theorem~\ref{thm:vector}]
\emph{Step 1: \textup{(i)} $\Rightarrow$ \textup{(ii)}.}
The uncertainty class contains the linear system
\[
 f(x_1,x_2)=L_1x_1+L_2x_2,\qquad B=I_n.
\]
Thus \textup{(i)} implies \textup{(ii)}.

\emph{Step 2: \textup{(ii)} $\Rightarrow$ \textup{(iii)}.}
For the linear system in \textup{(ii)}, the coordinates are
decoupled, and each satisfies the scalar error equation
\eqref{eq:linearerror}.
The Routh--Hurwitz criterion therefore requires
\[
 k_i>0,\qquad k_p>L_1,\qquad
 (k_p-L_1)(k_d-L_2)>k_i,
\]
which is precisely \textup{(iii)}.

\emph{Step 3: \textup{(iii)} $\Rightarrow$ \textup{(i)}.}
Assume \textup{(iii)}, and fix an admissible pair $(f,B)$
and any setpoint $y^*\in\R^n$.
Introduce
\begin{equation}\label{eq:vectortransform}
 \begin{gathered}
 x=\eta+\frac1{k_i}B^{-1}f(y^*,0),\qquad
 y=e,\qquad z=\dot e,\\
 g(y,z)=-f(y^*-y,-z)+f(y^*,0).
 \end{gathered}
\end{equation}
The error dynamics become
\begin{equation}\label{eq:vectorchain}
 \dot x=y,\qquad \dot y=z,\qquad
 \dot z=-k_iBx-k_pBy-k_dBz+g(y,z).
\end{equation}
Set
\[
 b=k_d-L_2>0,\qquad
 r=\frac{k_i}{b}>0,\qquad
 \delta=k_p-L_1-r>0,
\]
and define
\begin{equation}\label{eq:vectorN}
 \begin{aligned}
 s&=-z-rBx,\\
 N(y,z)&=(k_p-r)By+(k_dB-bI_n)z-g(y,z).
 \end{aligned}
\end{equation}
Since $br=k_i$, the transformed equations are
\[
 \dot y=z,\qquad
 \dot z=-rBy+bs-N(y,z),\qquad
 \dot s=-bs+N(y,z),
\]
which have the form \eqref{eq:normalform}.
Moreover, $N(0,0)=0$, and the assumptions on $f$ and $B$ give
\begin{align}
 D_yN
 &=(k_p-r)B-D_yg
 \succeq(k_p-r-L_1)I_n
 =\delta I_n,\notag\\
 D_zN
 &=k_dB-bI_n-D_zg
 \succeq(k_d-b-L_2)I_n
 =0.
 \label{eq:vectorNbounds}
\end{align}
Both partial Jacobians are symmetric.
Thus all the hypotheses of Lemma~\ref{lem:normalform}
are satisfied, and the origin in $(y,z,s)$ is globally
asymptotically stable.
Since the coordinate transformation is invertible,
the original closed-loop system achieves global regulation
for every initial state.
This proves \textup{(i)} and completes the equivalence.
\end{proof}

\begin{proof}[Proof of Proposition~\ref{prop:symmetry}]
Write
\[
 K_\nu=\operatorname{diag}(k_{\nu,1},\ldots,k_{\nu,n}),
 \qquad \nu\in\{p,i,d\}.
\]
Both enlarged classes contain the boundary linear system
$f=L_1x_1+L_2x_2$, $B=I_n$.
Under diagonal PID gains, its coordinates evolve independently.
The scalar linear test in Theorem~\ref{thm:scalar} therefore
forces every candidate gain triple to satisfy
\[
 (k_{p,\ell},k_{i,\ell},k_{d,\ell})\in\Opid,
 \qquad \ell=1,\ldots,n.
\]
Otherwise, this boundary system already gives a counterexample.

We first construct counterexamples for $n=2$. Set
\[
 a_\ell=k_{p,\ell}-L_1>0,\qquad
 d_\ell=k_{d,\ell}-L_2>0,\qquad
 c_\ell=k_{i,\ell}>0,\qquad \ell=1,2,
\]
and define
\[
 p_\ell(\lambda)
 =\lambda^3+d_\ell\lambda^2+a_\ell\lambda+c_\ell,
 \qquad
 J=\begin{pmatrix}0&-1\\1&0\end{pmatrix}.
\]
For a linear plant $f(x_1,x_2)=A_1x_1+A_2x_2$ with $B=I_2$,
the shifted variables
\[
 X=\eta+K_i^{-1}A_1y^*,\qquad Y=e,\qquad Z=\dot e
\]
satisfy
\[
 \dot X=Y,\qquad \dot Y=Z,\qquad
 \dot Z=-K_iX-(K_p-A_1)Y-(K_d-A_2)Z.
\]
Thus the closed-loop characteristic polynomial is
\[
 \det\!\left(
 \lambda^3I_2+(K_d-A_2)\lambda^2
 +(K_p-A_1)\lambda+K_i
 \right).
\]

\emph{Removing the symmetry of $D_1f$.}
For $m>0$, choose
\[
 A_1=L_1I_2+mJ,\qquad A_2=L_2I_2.
\]
Since $\operatorname{sym}A_1=L_1I_2$ and $A_2=L_2I_2$,
this plant belongs to $\mathcal C_1$.
Its characteristic polynomial is
\[
 P_m(\lambda)
 =\det\begin{pmatrix}
 p_1(\lambda)&m\lambda\\
 -m\lambda&p_2(\lambda)
 \end{pmatrix}
 =p_1(\lambda)p_2(\lambda)+m^2\lambda^2.
\]
With $\lambda=\sqrt m\,\mu$, we obtain
\[
 m^{-3}P_m(\sqrt m\,\mu)
 \longrightarrow \mu^6+\mu^2
 =\mu^2(\mu^4+1)
 \qquad\text{as }m\to\infty.
\]
The limiting polynomial has the simple root
$\mu_0=(1+i)/\sqrt2$, whose real part is positive.
Continuity of a simple root gives roots $\mu_m\to\mu_0$
of the scaled polynomial. Hence
$\lambda_m=\sqrt m\,\mu_m$ is a closed-loop eigenvalue
with positive real part for all sufficiently large $m$.

\emph{Removing the symmetry of $D_2f$.}
Choose instead
\[
 A_1=L_1I_2,\qquad A_2=L_2I_2+mJ.
\]
Now $A_1=L_1I_2$ and $\operatorname{sym}A_2=L_2I_2$,
so this plant belongs to $\mathcal C_2$.
Its characteristic polynomial becomes
\[
 Q_m(\lambda)
 =\det\begin{pmatrix}
 p_1(\lambda)&m\lambda^2\\
 -m\lambda^2&p_2(\lambda)
 \end{pmatrix}
 =p_1(\lambda)p_2(\lambda)+m^2\lambda^4.
\]
With $\lambda=m^{-1/2}\mu$, we have
\[
 Q_m(m^{-1/2}\mu)
 \longrightarrow c_1c_2+\mu^4
 \qquad\text{as }m\to\infty.
\]
The limiting polynomial has the simple root
\[
 \mu_0=(c_1c_2)^{1/4}\frac{1+i}{\sqrt2},
\]
whose real part is positive.
The same simple-root argument yields roots $\mu_m\to\mu_0$,
so $\lambda_m=m^{-1/2}\mu_m$ is an unstable closed-loop
eigenvalue for all sufficiently large $m$.

These unstable eigenvalues prevent regulation even with
$\eta(0)=0$. Indeed, $A_1$ is invertible in both examples,
and any real shifted initial state $(v_0,v_1,v_2)$ can be
realized by choosing
\[
 y^*=A_1^{-1}K_iv_0,\qquad
 x_1(0)=y^*-v_1,\qquad
 x_2(0)=-v_2,\qquad
 \eta(0)=0.
\]
Every eigenvector associated with a nonzero eigenvalue has
a nonzero error component, as follows from $\dot X=Y$
and $\dot Y=Z$.
An unstable mode therefore gives a real initial state for
which the regulation error does not converge to zero.

For $n>2$, place either counterexample in the first two
coordinates and append
\[
 f_\ell=L_1x_{1,\ell}+L_2x_{2,\ell},
 \qquad \ell=3,\ldots,n.
\]
The resulting plant remains in the corresponding class.
Since the gains are diagonal, the first two coordinates
remain independent of the others and retain the unstable mode.
Thus every fixed diagonal gain triple fails on a linear
member of each enlarged class.
\end{proof}

\section{Conclusion}\label{sec:conclusion}

In this paper, we have developed a mathematical theory and an
explicit gain design method for classical PID regulation of a basic
class of second-order nonaffine uncertain systems.
The main analytical tool is a structured Markus--Yamabe theorem,
proved by combining a two-dimensional storage construction through
Hamilton--Jacobi viscosity solutions with scalar comparison.
This theorem establishes the sufficiency of the PID gain conditions
imposed by a boundary linear system, while inclusion of that system
in the uncertainty class gives necessity.
Consequently, for each fixed PID gain triple, global regulation of
the entire nonlinear class is equivalent to global regulation of
this single linear system.

The resulting design conditions are explicit, necessary and
sufficient, and depend only on the prescribed derivative bounds.
They guarantee global  stability and asymptotic regulation
for all admissible plant, constant reference, and initial condition.
We have also extended the same exact characterization of
scalar PID gains to multidimensional affine systems under
the stated symmetry assumptions.
If either state-Jacobian symmetry assumption is removed while
the corresponding skew-symmetric part is left unrestricted,
no fixed diagonal PID controller guarantees global regulation
of the resulting enlarged class, even with $B=I_n$.

For future investigation, it would be interesting to extend
the single-linear-test principle to more general nonlinear
uncertainty classes and to determine whether the derivative
conditions in our structured Markus--Yamabe theorem
can be weakened.
Further questions include exact PID gain characterizations in the
presence of input delays and under sampled-data implementation
with a prescribed sampling rate.

\newpage
\section*{Appendix}

\subsection{ Hamilton--Jacobi initial value problem}
\label{app:HJ}

We record the comparison and existence result used to construct
the storage function in Lemma~\ref{lem:storage}, together with
the viscosity-solution terminology needed below.

Let $U:\mathbb R\times[0,T]\to\mathbb R$ be continuous.
Fix an interior point $(\xi_0,s_0)\in\mathbb R\times(0,T)$. A $C^1$ function $\varphi$, defined in a neighborhood of this point, is said to touch $U$ from above (from below) at $(\xi_0,s_0)$ if $\varphi(\xi_0,s_0)=U(\xi_0,s_0)$
and $U\le\varphi$ ($U\ge\varphi$) in some neighborhood of the point.
Such a function is called a test function.
If, for every such interior point and every test function $\varphi$ touching $U$
from above (respectively, from below), we have $\varphi_s(\xi_0,s_0)
 +H\bigl(\xi_0,s_0,\varphi_\xi(\xi_0,s_0)\bigr)\le0$ (respectively, $\ge0$),
then $U$ is called a viscosity subsolution
(respectively, a viscosity supersolution) of the equation $U_s+H(\xi,s,U_\xi)=0$.
A continuous function that is both a viscosity subsolution and a viscosity supersolution is called a viscosity solution.
This definition does not require $U$ to be differentiable.

\begin{lemma}\label{lem:HJ}
Let $H\in C^1(\R\times[0,\infty)\times\R)$. Suppose that there exist constants $L,M>0$ such that, for every $T>0$, there is a constant $A_T<\infty$ for which $|H_q(\xi,s,q)|\le L$, $|H_\xi(\xi,s,q)|\le A_T(1+|q|)$, and $-Ms\le H(\xi,s,0)\le0$ hold for all $\xi,q\in\R$ and $0\le s\le T$.
Write $Q_T=\R\times[0,T]$. Throughout, spatial uniform continuity means uniform continuity with respect to $\xi$, uniformly for $s\in[0,T]$. Then the following statements hold.
\begin{enumerate}
\renewcommand{\labelenumi}{(\roman{enumi})}
\item The initial value problem $U_s+H(\xi,s,U_\xi)=0$ for $s>0$, with $U(\xi,0)=0$, admits a continuous viscosity solution $U:\R\times[0,\infty)\to\R$ such that $0\le U(\xi,s)\le Ms^2/2$ and $U$ is spatially uniformly continuous on every $Q_T$.
\item Let $u,v\in C(Q_T)$ be bounded viscosity sub- and supersolutions, respectively, of the equation on $\R\times(0,T)$. If at least one of them is spatially uniformly continuous and $u(\cdot,0)\le v(\cdot,0)$, then $u\le v$ on $Q_T$.
\end{enumerate}
\end{lemma}

The existence and comparison assertions follow from
the Hamilton--Jacobi theory in \cite{CrandallLions1986},
including the comparison result requiring spatial uniform
continuity of only one of the two functions.
The bounds in \textup{(i)} follow by comparison with
the subsolution $0$ and the supersolution $Ms^2/2$.
Solutions on finite time intervals agree by uniqueness
and therefore define a solution for all $s\ge0$.

\subsection{Construction of the storage function}
\label{app:storage}
This appendix states and proves the two-dimensional storage lemma
used in Theorem~\ref{thm:MY}.

\begin{lemma}\label{lem:storage}
Fix $\mu,\delta,\varepsilon>0$.
Let $B\in C^1(\R^3)$ satisfy
$B_p\ge1$, $B_y\ge\delta B_p$, and
$B_v\ge\varepsilon B_p$ on $\R^3$.
Then there exists a locally Lipschitz function $W:\R^2\to\R$ such that
\begin{equation}
 0\le W(y,p)\le \frac{[p-p_0(y)]^2}{2\mu\varepsilon},
 \label{eq:storage-value}
\end{equation}
where $p_0(y)$ is the unique solution in $p$ of the equation $B(y,0,p)=0$. Moreover, at every point where $W$ is differentiable,
\begin{equation}
 W_y=B(y,-\mu W_p,p),\qquad W_pW_y\ge\delta W_p^2.
 \label{eq:storage-match}
\end{equation}
\end{lemma}

\begin{proof}[Proof of Lemma~\ref{lem:storage}]

Throughout the proof, $B_y,B_v,B_p$ denote partial derivatives with
respect to the independent variables $(y,v,p)$, and
$\|\cdot\|_\infty$ denotes the supremum norm over the entire domain.

\step{Step 1: Determination of the zero curve.}
Fix $y\in\R$. Since $B_p\ge1$, the mean value theorem gives
\[
 \begin{aligned}
 B(y,0,p)\ge B(y,0,0)+p~(p\ge0),~B(y,0,p)\le B(y,0,0)+p~(p\le0).
 \end{aligned}
\]
Thus, for each fixed $y\in\R$, the function $p\mapsto B(y,0,p)$ is strictly increasing,
and $\lim_{p\to+\infty}B(y,0,p)=+\infty,\lim_{p\to-\infty}B(y,0,p)=-\infty$.
Hence the equation $B(y,0,p)=0$ has a unique solution in $p$, denoted by $p_0(y)$.
This defines a function $p_0:\R\to\R$ satisfying
$B(y,0,p_0(y))=0,~y\in\R.$
The implicit function theorem and uniqueness of the solution imply that $p_0\in C^1(\R)$ and
$p_0'(y)
 =-B_y(y,0,p_0(y))/B_p(y,0,p_0(y))\le-\delta.$
Set $c_0(y):=-p_0'(y)\ge\delta$.

\medskip
\emph{In Steps 2--5 below, we assume that $B\in C^\infty(\R^3)$
and that its first- and second-order partial derivatives are globally bounded.}
Then $p_0$ is smooth and
$\delta\le c_0(y)\le\|B_y\|_\infty$,
so $p_0$ is globally Lipschitz. A direct differentiation also shows that $c_0'$ is globally bounded.

\step{Step 2: Construction of the Hamiltonian and verification of the hypotheses of Lemma~\ref{lem:HJ}.}

The zero set of the function $(y,p)\mapsto B(y,0,p)$ is the curve
$\Gamma_0:=\{(y,p_0(y)):y\in\R\}.$
Set $\Omega_+:=\{(y,p)\in\R^2:p>p_0(y)\}.$
Introduce the smooth change of coordinates
\[
 (y,p)\mapsto(y,n),\qquad n=p-p_0(y),
\]
whose inverse is $(y,n)\mapsto(y,p_0(y)+n)$.
This transformation maps $\Gamma_0$ onto the line $n=0$
and $\Omega_+$ onto the open half-plane $\R\times(0,\infty)$.

We next construct a Hamiltonian to which Lemma~\ref{lem:HJ} applies.
Introduce independent real variables $\zeta,\eta$ representing $U_y,U_n$, respectively, and define the auxiliary function
\[
 G(y,n,\zeta,\eta)
 :=\zeta+c_0(y)\eta-B(y,-\mu\eta,p_0(y)+n).
\]
Since
\[
 G_\eta
 =c_0(y)+\mu B_v(y,-\mu\eta,p_0(y)+n)
 \ge\delta+\mu\varepsilon>0,
\]
for every fixed $(y,n,\zeta)\in\R^3$,
the function $\eta\mapsto G(y,n,\zeta,\eta)$ is strictly increasing
and tends to $\pm\infty$ as $\eta\to\pm\infty$, respectively.
It therefore has a unique zero, denoted by $A(y,n,\zeta)$. Define
\[
 H(y,n,\zeta):=-A(y,n,\zeta).
\]
The implicit function theorem and uniqueness of the zero imply that
$A,H\in C^\infty(\R^3)$.

In what follows, all partial derivatives of $B$ are evaluated at
$(y,-\mu A(y,n,\zeta),p_0(y)+n)$.

Differentiating the identity $\zeta+c_0(y)A-B(y,-\mu A,p_0(y)+n)=0$
gives that
$A_\zeta=-\frac1{c_0(y)+\mu B_v},~
 0<|A_\zeta|\le\frac1{\delta+\mu\varepsilon},$
\begin{equation}
 0<A_n=\frac{B_p}{c_0(y)+\mu B_v}\le \frac{B_p}{\delta+\mu\varepsilon B_p}\le \frac{1}{\mu\varepsilon}
 \label{eq:app-An}
\end{equation}
and
\begin{equation}
 A_y=\frac{B_y-c_0B_p-c_0'A}{c_0(y)+\mu B_v}.
 \label{eq:app-Ay}
\end{equation}
Since $B(y,0,p_0(y))=0$, uniqueness of the zero yields
$A(y,0,0)=0$.
Fix $y\in\R$ and take $\zeta=0$. By \eqref{eq:app-An}, for $n\ge0$ we have
$0\le A(y,n,0)
=\int_0^n A_n(y,t,0)\,\mathrm dt
\le n/(\mu\varepsilon)$.
Using $|A_\zeta|\le1/(\delta+\mu\varepsilon)$
and the mean value theorem in $\zeta$, we obtain
$|A(y,n,\zeta)|
 \le n/(\mu\varepsilon)
       +|\zeta|/(\delta+\mu\varepsilon),~y,~\zeta\in\R,~n\ge0.$
Consequently, $-n/(\mu\varepsilon)\le H(y,n,0)\le0$ for $n\ge0$.
Since $B_y,B_p,c_0,c_0'$ are globally bounded, \eqref{eq:app-Ay} and the preceding estimate for $A(y,n,\zeta)$ imply that, for every $T>0$, there exists a constant $C_T<\infty$ such that
$|H_y(y,n,\zeta)|\le C_T(1+|\zeta|)$
for all $y,\zeta\in\R$ and $0\le n\le T$.

Thus $H$ satisfies all the hypotheses of Lemma~\ref{lem:HJ},
with $(\xi,s,q)$ corresponding to $(y,n,\zeta)$, respectively.

\step{Step 3: Solution of the Hamilton--Jacobi equation and return to the original coordinates.}
By Lemma~\ref{lem:HJ}(i), the initial value problem
\begin{equation}
 U_n+H(y,n,U_y)=0\quad(n>0),\qquad U(y,0)=0
 \label{eq:app-HJ}
\end{equation}
admits a continuous viscosity solution $U:\R\times[0,\infty)\to\R$ satisfying
\begin{equation}
 0\le U(y,n)\le\frac{n^2}{2\mu\varepsilon},~n\ge0.
 \label{eq:app-barriers}
\end{equation}
For every $T>0$, $U$ is uniformly continuous with respect to $y$, uniformly for $n\in[0,T]$.

Define $$W_+(y,p):=U(y,p-p_0(y)),$$ for $p\ge p_0(y)$.
We now use test functions to verify that $W_+$ satisfies the corresponding equation in the viscosity sense on $\Omega_+$.

Since $G(y,n,\zeta,\eta)$ is strictly increasing in $\eta$ and has the unique zero $A(y,n,\zeta)=-H(y,n,\zeta)$, we have
\begin{equation}
 \eta+H(y,n,\zeta)\le0
 \quad\Longleftrightarrow\quad
 G(y,n,\zeta,\eta)\le0,
 \label{eq:app-sign}
\end{equation}
and the equivalence also holds with both inequalities reversed.

Define the continuous operator $\mathcal F(y,p,h,\alpha):=h-B(y,-\mu\alpha,p)$,
where $h,\alpha$ are independent gradient variables; in the test function characterization, they are replaced by $\varphi_y,\varphi_p$, respectively, at the point of contact.
We use the test function characterization of viscosity solutions~\cite[Chapter~1, Definition~1.1]{Tran2021} to verify that $W_+$ is a viscosity solution of $\mathcal F(y,p,(W_+)_y,(W_+)_p)=0$ on $\Omega_+$.

Suppose that $\varphi\in C^1$ touches $W_+$ from above at $(y_*,p_*)\in\Omega_+$.
Set $n_*:=p_*-p_0(y_*)$. Then $\Phi(y,n):=\varphi(y,p_0(y)+n)$ touches $U$ from above at $(y_*,n_*)$.
The chain rule gives $\Phi_y=\varphi_y-c_0(y)\varphi_p$ and $\Phi_n=\varphi_p$, with the partial derivatives of $\varphi$ evaluated at $(y,p_0(y)+n)$. The viscosity subsolution property of $U$ and \eqref{eq:app-sign} yield, at the point of contact,
$\varphi_y-B(y_*,-\mu\varphi_p,p_*)=G(y_*,n_*,\Phi_y,\Phi_n)\le0$.
The same argument for a test function touching from below gives the reverse inequality.

Thus $W_+$ is a viscosity solution of $\mathcal F(y,p,(W_+)_y,(W_+)_p)=0$ on $\Omega_+$.
We next prove that $W_+$ is locally Lipschitz.

\step{Step 4: Monotonicity and local Lipschitz continuity.}
First consider a translation $(h_0,k_0)$ satisfying
\begin{equation}
 B(y+h_0,v,p+k_0)\ge B(y,v,p),\qquad (y,v,p)\in\R^3
 \label{eq:app-translation}
\end{equation}
Taking $v=0$ and $p=p_0(y)$, we obtain $\beta(y):=p_0(y)+k_0-p_0(y+h_0)\ge0$ from $B_p>0$. Since $p_0$ is globally Lipschitz, $\beta$ is bounded, and $\beta'=c_0(y+h_0)-c_0(y)$.

Set $\widehat U(y,n)=U(y+h_0,n+\beta(y))$. If $\varphi$ touches $\widehat U$ from below at a point with $n>0$, then the transformed test gradient is $(\varphi_y-\beta'\varphi_n,\varphi_n)$. The supersolution property of $U$, \eqref{eq:app-sign}--\eqref{eq:app-translation}, and the identity $c_0(y+h_0)-\beta'=c_0(y)$ give $G(y,n,\varphi_y,\varphi_n)\ge0$. Thus $\widehat U$ is a viscosity supersolution of \eqref{eq:app-HJ}.
For each $T>0$, $U,\widehat U$ are continuous and bounded on
$\R\times[0,T]$,
and $U$ is uniformly continuous with respect to $y$, uniformly for $n\in[0,T]$. Moreover, $\widehat U(y,0)\ge0=U(y,0)$. The comparison principle in Lemma~\ref{lem:HJ} therefore yields $W_+(y+h_0,p+k_0)\ge W_+(y,p)$.

Let $e_1=(0,1)$ and $e_2=(1,-\delta)$. Since $B_p\ge1$ and $B_y-\delta B_p\ge0$, the translations $\ell e_i$ with $\ell>0$ and $i=1,2$ satisfy \eqref{eq:app-translation}, so $W_+$ is nondecreasing in these two directions. If $\varphi$ touches $W_+$ from above, write $\alpha=\varphi_p$ and $h=\varphi_y$. Directional monotonicity and the viscosity subsolution condition give $0\le\delta\alpha\le h\le B(y,-\mu\alpha,p)\le B(y,0,p)-\mu\varepsilon\alpha$. Hence $0\le\alpha\le B(y,0,p)/(\delta+\mu\varepsilon)$ and $0\le h\le B(y,0,p)$.

By localizing the cone comparison argument in~\cite[Lemma 1.33]{Tran2021}, the local boundedness of the gradients of test functions touching from above shows that $W_+$ is locally Lipschitz on $\Omega_+$.

Furthermore, for $0<n<T$, we have $0\le B(y,0,p_0(y)+n)\le\|B_p\|_\infty T$. Together with boundedness of $c_0$ and the transformation of test gradients $(\varphi_y,\varphi_p)=(\Phi_y+c_0\Phi_n,\Phi_n)$, this shows that the gradients of test functions touching $U$ from above are uniformly bounded on $\R\times(0,T)$. The same cone comparison argument and convexity of $\R\times(0,T)$ imply that $U$ is Lipschitz on $\R\times(0,T)$. By continuity, $U$ is also Lipschitz on $\R\times[0,T]$.

At every point where $W_+$ is differentiable, the viscosity equation holds in the classical sense~\cite[Remark~8.4]{Bressan2007}, while directional monotonicity gives $(W_+)_p\ge0$ and $(W_+)_y-\delta(W_+)_p\ge0$. Thus \eqref{eq:storage-match} holds.

\step{Step 5: Construction of a storage function on the whole plane.}
Define $\widetilde B(Y,V,P)=-B(-Y,-V,-P)$.
This function satisfies the same structural conditions and derivative bounds as $B$,
and the unique root of $\widetilde B(Y,0,P)=0$ is
$\widetilde p_0(Y)=-p_0(-Y)$.
Applying Steps 2--4 to $\widetilde B$ yields $\widetilde W_+$.
On $\Omega_-:=\{(y,p):p<p_0(y)\}$, define
$W_-(y,p)=\widetilde W_+(-y,-p)$.

By this change of variables and the conclusions of Steps 2--4,
$W_-$ is a viscosity solution of $-\mathcal F=0$ on $\Omega_-$,
is nonincreasing in the directions $e_1,e_2$ within $\Omega_-$, and satisfies \eqref{eq:storage-value}.
Consequently, at every point where $W_-$ is differentiable,
we have $(W_-)_y=B(y,-\mu(W_-)_p,p)$,
$(W_-)_p\le0$, and
$(W_-)_y-\delta(W_-)_p\le0$.
Thus \eqref{eq:storage-match} holds.

Define $W$ to equal
$W_+$ on $\Omega_+$, $W_-$ on $\Omega_-$, and zero on $\Gamma_0$.
For any $T>0$, in the coordinates $n=p-p_0(y)$,
Step 4 shows that the functions on the two sides are Lipschitz on
$\R\times[0,T]$ and $\R\times[-T,0]$, respectively,
and both vanish on $n=0$.
For any two points on opposite sides, split the line segment joining them at $n=0$
to obtain a Lipschitz estimate for the combined function.
Returning to the original coordinates, we conclude that $W$ is locally Lipschitz on $\R^2$
and satisfies \eqref{eq:storage-value}.

For any $x_*\in\Gamma_0$,
local Lipschitz continuity of $p_0$ and \eqref{eq:storage-value} give
$W(x)=O(|x-x_*|^2)$.
Hence $W$ is differentiable at $x_*$ and $DW(x_*)=0$.
Since $B(y,0,p_0(y))=0$,
\eqref{eq:storage-match} also holds on the zero curve.

\step{Step 6: Smooth approximation preserving the structural conditions.}
\emph{We now return to the general case $B\in C^1(\R^3)$}
and construct smooth approximations $B_j:\R^3\to\R$
that satisfy the same structural conditions and have globally bounded first- and second-order partial derivatives.

Introduce the change of coordinates $s=p+\delta y+\varepsilon v$
and define
$\psi(y,v,s)=B(y,v,s-\delta y-\varepsilon v)-s$.
Here $(y,v,s)$ are the independent variables of $\psi$,
and the original structural conditions are equivalent to
$\psi_y(y,v,s),\psi_v(y,v,s),\psi_s(y,v,s)\ge0$.

Let $C_c^\infty(\R^d)$ denote the space of smooth functions with compact support in $\R^d$.
Choose $\rho\in C_c^\infty(\R^3)$ such that
$\rho\ge0$ and $\int_{\R^3}\rho(x)\,\mathrm dx=1$.
For $j=1,2,\ldots$, set $\rho_j(x)=j^3\rho(jx)$
and define $\psi_j(y,v,s)=(\rho_j*\psi)(y,v,s)$.
Then $\psi_j$ is smooth, all its first-order partial derivatives are nonnegative,
and $\psi_j\to\psi$ in $C^1_{\mathrm{loc}}(\R^3)$; that is, the functions and all their first-order partial derivatives converge uniformly on every compact subset of $\R^3$.

Choose $\chi\in C_c^\infty(\R)$ with $0\le\chi\le1$,
such that $\chi(t)=1$ for $|t|\le1$
and $\chi(t)=0$ for $|t|\ge2$,
and set $\sigma_j(t)=\int_0^t\chi(u/j)\,\mathrm du$.
Define
$B_j(y,v,p)
=s+\psi_j(\sigma_j(y),\sigma_j(v),\sigma_j(s))$,
where $s=p+\delta y+\varepsilon v$.
Since the first-order partial derivatives of $\psi_j$ and the derivative $\sigma_j'$ are all nonnegative,
for every $(y,v,p)\in\R^3$ we have
$(B_j)_p(y,v,p)\ge1$,
$(B_j)_y(y,v,p)\ge\delta(B_j)_p(y,v,p)$, and
$(B_j)_v(y,v,p)\ge\varepsilon(B_j)_p(y,v,p)$.

For each fixed $j$, the function $\sigma_j$ and all its derivatives are bounded.
Hence $B_j$ is smooth and its first- and second-order partial derivatives are globally bounded.
Moreover, since $\sigma_j(t)=t$ for $|t|\le j$,
and $\psi_j$ converges to $\psi$ in $C^1_{\mathrm{loc}}(\R^3)$,
we obtain convergence in $C^1_{\mathrm{loc}}(\R^3)$:
\begin{equation}
 B_j\longrightarrow B.
 \label{eq:app-Bapprox}
\end{equation}
Let $p_{0,j}(y)$ be the unique root in $p$ of $B_j(y,0,p)=0$.
Since $(B_j)_p\ge1$, for every $y\in\R$ we have
$|p_{0,j}(y)-p_0(y)|
\le|B_j(y,0,p_0(y))-B(y,0,p_0(y))|$.
Together with \eqref{eq:app-Bapprox}, this implies that the following convergence is uniform on every compact interval in $\R$:
\begin{equation}
 p_{0,j}\longrightarrow p_0.
 \label{eq:app-rootapprox}
\end{equation}
For each $j$, apply Steps 2--5 to $B_j(y,v,p)$
to obtain the corresponding storage function $W_j(y,p)$.

\step{Step 7: Compactness and passage to the limit.}
By \eqref{eq:storage-match} applied to $B_j,W_j$,
the inequality $(B_j)_v\ge\varepsilon$, and the mean value theorem in $v$,
at every point $(y,p)$ where $W_j$ is differentiable we have
\begin{equation}
 |(W_j)_p(y,p)|
 \le\frac{|B_j(y,0,p)|}{\delta+\mu\varepsilon},
 \qquad
 |(W_j)_y(y,p)|\le|B_j(y,0,p)|.
 \label{eq:app-gradientbounds}
\end{equation}
By \eqref{eq:app-Bapprox}, these gradients are bounded uniformly in $j$ on every compact set.
Since $W_j$ is locally Lipschitz,
its weak gradient agrees almost everywhere with its classical gradient.
Thus $\{W_j\}$ is locally equi-Lipschitz.
The quadratic bound
$0\le W_j(y,p)\le[p-p_{0,j}(y)]^2/(2\mu\varepsilon)$
and \eqref{eq:app-rootapprox} also show that $\{W_j\}$ is locally uniformly bounded.
By the Arzel\`a--Ascoli theorem and a diagonal argument,
there is a subsequence, still denoted by $W_j$,
such that $W_j\to W$ uniformly on every compact subset of $\R^2$.
The uniform local Lipschitz estimates and the quadratic bounds pass to the limit.
Hence $W$ is locally Lipschitz on $\R^2$ and satisfies \eqref{eq:storage-value}.

Set
$\Omega_{\pm,j}
=\{(y,p)\in\R^2:\pm(p-p_{0,j}(y))>0\}$
and define
$\mathcal F_j(y,p,h,\alpha)
=h-B_j(y,-\mu\alpha,p)$.
By \eqref{eq:app-Bapprox}--\eqref{eq:app-rootapprox},
$\mathcal F_j\to\mathcal F$ locally uniformly in
$(y,p,h,\alpha)\in\R^4$,
and every compact subset of $\Omega_\pm$ is contained in
$\Omega_{\pm,j}$ for all sufficiently large $j$.
On any fixed open set compactly contained in $\Omega_\pm$,
apply the stability of viscosity solutions
\cite[Section~8.4]{Bressan2007}
to conclude that $W$ is a viscosity solution of
$\mathcal F(y,p,W_y,W_p)=0$ on $\Omega_+$ and of
$-\mathcal F(y,p,W_y,W_p)=0$ on $\Omega_-$.
Therefore, at every point $(y,p)$ in $\Omega_+\cup\Omega_-$
where $W$ is differentiable, we have
$W_y(y,p)=B(y,-\mu W_p(y,p),p)$.

Similarly, local uniform convergence and \eqref{eq:app-rootapprox} imply that
the directional monotonicity of $W_j$ along $e_1,e_2$
passes to $W$ in the interior of each region.
Thus, at every point $(y,p)$ where $W$ is differentiable,
we have
$W_p(y,p)\ge0$ and
$W_y(y,p)-\delta W_p(y,p)\ge0$ on $\Omega_+$,
while both inequalities are reversed on $\Omega_-$.
Consequently,
$W_p(y,p)\bigl(W_y(y,p)-\delta W_p(y,p)\bigr)\ge0$ on both sides.
Finally, \eqref{eq:storage-value} and the quadratic bound argument in Step 5 show that
$W$ is differentiable at every point of $\Gamma_0$,
with $DW(y,p_0(y))=0$ for all $y\in\R$.
Since $B(y,0,p_0(y))=0$,
\eqref{eq:storage-match} holds at every point where $W$ is differentiable.
\end{proof}

\end{document}